\documentclass{amsart}
\usepackage{amsthm}
\usepackage{amssymb}
\usepackage{upref}
\usepackage{amscd}

\numberwithin{equation}{section}

\newtheorem{theorem}[equation]{Theorem}
\newtheorem{lemma}[equation]{Lemma}
\newtheorem{proposition}[equation]{Proposition}
\newtheorem{corollary}[equation]{Corollary}

\theoremstyle{definition}
\newtheorem{definition}[equation]{Definition}

\DeclareMathOperator{\Hom}{Hom}

\DeclareMathOperator{\Tor}{Tor}

\DeclareMathOperator{\colim}{colim}

\DeclareMathOperator{\im}{im}

\DeclareMathOperator{\coker}{coker}

\DeclareMathOperator{\ghdim}{gh.dim.}

\DeclareMathOperator{\wdim}{w. dim.}
\DeclareMathOperator{\Rdim}{Rouq. dim.}

\DeclareMathOperator{\pdim}{proj. dim.}
\DeclareMathOperator{\fdim}{flat \  dim.}
\DeclareMathOperator{\cfdim}{con.\ flat \  dim.}
\newcommand{\class}[1]{\langle #1 \rangle}

\newcommand{\Q}{\mathbb{Q}}

\newcommand{\cat}[1]{\mathcal{#1}}

\newcommand{\Z}{\mathbb{Z}}

\newcommand{\mathcolon}{\colon\,}

\newcommand{\uc}{\textup{:}}
\newcommand{\ulp}{\textup{(}}
\newcommand{\urp}{\textup{)}}

\begin{document}

\title{The ghost and weak dimensions of rings and ring spectra} 

\date{\today}

\author{Mark Hovey}
\address{Department of Mathematics \\ Wesleyan University
\\ Middletown, CT 06459}
\email{hovey@member.ams.org}

\author{Keir Lockridge}
\address{Department of Mathematics \\ Wake Forest University \\
Winston-Salem, NC 27109}
\email{lockrikh@wfu.edu}

%\subjclass{55P43, 16E10, 13D05}

\begin{abstract}
The primary object of this paper is to prove the conjecture of
\cite{hovey-lockridge-ghost} explaining how to recover the weak
dimension of a ring from its derived category.  In the process, we
develop a theory of weak dimension, which we call ghost dimension, for
the generalized rings, known as ring spectra, that arise in
algebraic topology.
\end{abstract}

\maketitle

\section*{Introduction}

In a previous paper~\cite{hovey-lockridge-ghost}, the authors
considered the problem of recovering the weak dimension of a ring $R$
from the derived category $\cat{D} (R)$, together with its
distinguished object $R$.  In that paper, the authors defined the
\textbf{ghost dimension} of $R$, $\ghdim R$, and proved that $\ghdim
R\geq \wdim R$, with equality holding when $R$ is coherent or has weak
dimension $1$.  In the present paper, we prove that $\ghdim R=\wdim R$
for all rings $R$.  

The point of doing this, besides its intrinsic interest, is to allow
consideration of weak dimension for more general kinds of rings.  In
algebraic topology, for example, there is a notion of a ring spectrum,
or, more precisely, an $S$-algebra
$E$~\cite{elmendorf-kriz-mandell-may}.  Such an $S$-algebra has no
elements in the usual sense.  There is a category of (right)
$E$-modules, but it is not abelian.  However, there is a derived
category $\cat{D} (E)$ of $E$, and it shares many of the formal
properties of the derived category $\cat{D} (R)$ of an ordinary ring
$R$; in particular, $\cat{D} (E)$ is a compactly generated
triangulated category, and there are derived tensor products and
derived Hom objects.  In fact, every ring $R$ has an associated
Eilenberg-MacLane $S$-algebra $HR$, and $\cat{D} (HR)$ is equivalent
to $\cat{D} (R)$.  To define invariants of such $S$-algebras $E$,
then, one way to proceed is to define usual ring invariants, such as
the weak or global dimension, in terms of $\cat{D} (R)$, and then
apply this definition to $\cat{D} (E)$ as well.

The second author did this for the (right) global dimension in his
thesis, and we now summarize this.  For further details,
see~\cite{hovey-lockridge-ghost}.  Define a map $f\mathcolon
X\xrightarrow{}Y$ in $\cat{D} (E)$ to be a \textbf{ghost} if $\cat{D}
(E) (E,f)_{*}=0$.  In the case that $E$ is an ordinary ring $R$, a
ghost is then just a map that induces the zero homomorphism on
homology.  If $E$ is an $S$-algebra, a ghost is a map that induces the
zero homomorphism on homotopy groups.  The second author shows that
the right global dimension of a ring $R$ is the least $n$ for which
every composite of $n+1$ ghosts in $\cat{D} (R)$ is null, or $\infty$
if there is are arbitrarily long nonzero composites of ghosts.  We can
then define the global dimension of an $S$-algebra $E$ in an analogous
fashion.  The authors did this and investigated the $S$-algebras of
global dimension $0$ in~\cite{hovey-lockridge-semisimple}.

Weak dimension is more complicated, and there seem to be many possible
definitions.  A major goal of this paper is to elucidate the different
possibilities and to find the correct one.  The \textbf{ghost
dimension} of an $S$-algebra $E$ or a ring $R$ is the least $n$ such
that every composite
\[
X_{0}\xrightarrow{f_{1}}X_{1}\xrightarrow{f_{2}}\dotsb
\xrightarrow{f_{n+1}} X_{n+1}
\]
of $n+1$ ghosts in $\cat{D} (E)$ (or $\cat{D} (R)$), where $X_{0}$ is
a compact object, is null (or $\infty$ if there are arbitrarily long
nonzero composites of ghosts out of compact objects).  Recall that $X$
is compact in a triangulated category $\cat{C}$ if the functor
$\cat{C} (X,-)$ preserves coproducts.  In particular, the compact
objects of $\cat{D} (R)$ are the perfect complexes (complexes
quasi-isomorphic to bounded complexes of finitely generated
projectives), and the compact objects of $\cat{D} (E)$ are the
retracts of finite cell $E$-modules.  The ghost dimension of a ring
was discussed in~\cite{hovey-lockridge-ghost}, as mentioned above.

In addition, a version of weak dimension closely related to Rouquier's
definition of the dimension of a triangulated category~\cite{rouquier}
is defined in~\cite{hovey-lockridge-semisimple}.  Neither of these
uses the notion of a flat $E$-module.  This obvious oversight was made
because of a difficulty with flat modules that we now recall.  If $E$
is an $S$-algebra, and $F_{*}$ is a flat left $E_{*}=\pi_{*}E=\cat{D}
(E) (E,E)_{*}$-module, then we can form a homology theory (a
coproduct-preserving exact functor to abelian groups) on $\cat{D} (E)$
that takes $M$ to $M_{*}\otimes_{E_{*}}F_{*}$.  One would like to say
that Brown representability for homology theories then forces there to
be a left $E$-module $F$ with $\pi_{*}F\cong F_{*}$.  Unfortunately,
Brown representability does not hold for a general ring
spectrum~\cite{neeman-brown, christensen-keller-neeman}, so flat
modules may not be realizable.  This worrying phenomenon led us to
doubt the utility of flat modules.  However, we use them in this
paper.  One of the surprising things we discover is the following.
Call $X\in \cat{D} (E)$ \textbf{flat} if $X_{*}$ is a flat
$E_{*}$-module.  Then we prove that $X$ is flat if and only if every
ghost $f$ whose domain is $X$ is phantom, in the sense that $\cat{D}
(E) (A,f)=0$ for all compact $A\in \cat{D} (E)$.  This gives us two
different notions of flat dimension.  The one most similar to the
algebraic situation we call the \textbf{constructible flat dimension},
$\cfdim X$.  It is a measure of how many steps one needs to construct
$X$ from flat objects of $\cat{D} (E)$.  We reserve the term
\textbf{flat dimension}, $\fdim X$, for the smallest $n$ such that
every composite of $n+1$ ghosts with domain $X$ is phantom.  This
seems algebraically strange, but has better properties.  This gives
several more notions of weak dimension: the maximal constructible flat
(resp. flat) dimension of a compact $E$-module, and the maximal
constructible flat (resp. flat) dimension of an arbitrary $E$-module.
We show that all of these are equal to the ghost dimension, except
possibly the maximal constructible flat dimension of an arbitrary
$E$-module.  The conjecture of~\cite{hovey-lockridge-ghost} that
$\ghdim R=\wdim R$ then follows.

In the end, we are left with three definitions of weak dimension for
an $S$-algebra $E$.  There is $\ghdim E$, which coincides with the
maximal flat dimension of any object.  There is the maximal
constructible flat dimension of any object, which agrees with $\ghdim
E$ for $E=HR$, but possibly not in general.  And there is the Rouquier
dimension $\Rdim E$, which agrees with $\ghdim E$ when $E_{*}$ is
coherent.  We prove that the ghost dimension is right-left symmetric,
which we have been unable to do with any of the other definitions.
Hence we argue that the ghost dimension is the proper version of weak
dimension for $S$-algebras $E$.  

This subject sorely needs examples, in order to be sure that all these
definitions are in fact distinct.  It should be possible to find an
ordinary ring $R$ such that the Rouqier dimension of $\cat{D} (R)$ is
distinct from the other dimensions.  Such an example would involve
serious analysis of the derived category of a non-coherent ring.  To
determine whether the constructible flat dimension is different from
the flat dimension would seem to require a new idea.  

%% Can we get any further if we assume D (E) is a Brown category?  

Note that all modules we use in this paper are right modules unless
explicitly stated otherwise.  The reader who is interested only in
ordinary rings can read $R$ everywhere the symbol $E$ or $E_{*}$
appears, read ``chain complex of $R$-modules'' whenever the term
``$E$-module'' appears, and read $H_{*}X$ everywhere $X_{*}$ appears,
for $X$ an $E$-module.  

\section{Ghost dimension and Rouquier dimension}\label{sec-ghost}

For an $S$-algebra $E$ or a ring $R$, the authors have previously
considered two different possible definitions related to weak
dimension, which we now discuss.  First of all, we can define the
Rouquier dimension to be the maximum number of steps needed to build a
compact object of $\cat{D} (E)$ from finitely many copies of $E$
(along the lines of Rouquier~\cite{rouquier}).  In more detail, given
a class $\cat{A}$ of objects of $\cat{D} (E)$, define
$\class{\cat{A}}^{n}$ inductively as follows.  Define
$\class{\cat{A}}^{0}$ to be the collection of all retracts of
coproducts of suspensions of elements of $\cat{A}$, and define an
object $X$ to be in $\class{\cat{A}}^{n}$ if and only if it is a
retract of an object $\widetilde{X}$ for which there is an exact
triangle
\[
A \xrightarrow{}Y \xrightarrow{}\widetilde{X} \xrightarrow{} A
\]
where $A\in \class{\cat{A}}^{0}$, and $Y\in \class{\cat{A}}^{n-1}$.
If $\cat{A}$ is a class of compact objects, we define
$\class{\cat{A}}^{n}_{f}$ similarly, with $\class{\cat{A}}^{0}_{f}$ being
the collection of all retracts of \emph{finite} coproducts of
suspensions of elements of $\cat{A}$, and then using the same
induction procedure to define $\class{\cat{A}}^{n}_{f}$.  Then
$\class{\cat{A}}^{n}_{f}$ consists of compact objects in
$\class{\cat{A}}^{n}$, but there may be compact objects in
$\class{\cat{A}}^{n}$ that are nevertheless not in
$\class{\cat{A}}^{n}_{f}$.  

We define the \textbf{Rouquier dimension} of $E$ (or $R$), $\Rdim E$,
to be the smallest $n$ such that $\class{E}^{n}_{f}$ is all of the
compact objects, or $\infty$ if no such $n$ exists.  This was called
the weak dimension in \cite{hovey-lockridge-semisimple}, but that
seems inappropriate, since we do not know that it agrees with the weak
dimension when $E$ is an ordinary ring $R$.  We define the
\textbf{ghost dimension} of $E$ (or $R$), $\ghdim E$, to be the
smallest $n$ such that $\class{E}^{n}$ contains all the compact
objects.  We also define the \textbf{projective dimension}, $\pdim X$,
of a given object $X$ to be the smallest $n$ such that $X\in
\class{E}^{n}$.  This was called the \textbf{ghost length}
in~\cite{hovey-lockridge-ghost}.  Then $\ghdim E$ is the supremum of
$\pdim X$ for $X$ compact.

The following proposition explains the connection to the definition
given in the introduction. 

\begin{proposition}\label{prop-proj-dim}
Suppose $E$ is an $S$-algebra or an ordinary ring, and $X\in \cat{D}
(E)$.  Then $\pdim X\leq n$ if and only if every composite of $n+1$
ghosts with domain $X$ is the zero map.  Furthermore, $\pdim X\leq
\pdim_{E_{*}} X_{*}$, with equality when $\pdim X=0$ and also when $E$
is an ordinary ring and $X$ is the projective resolution of a module
$M$.
\end{proposition}

This proposition is the content of
Proposition~1.1, the proof of Proposition~1.3, and Lemma~1.4
of~\cite{hovey-lockridge-ghost}, although Proposition~1.1
of~\cite{hovey-lockridge-ghost} is really due to
Christensen~\cite[Theorem~3.5]{christensen}.

We commonly call the objects $P$ with $\pdim P=0$ \textbf{projective},
as this proposition implies $P$ is projective if and only if $P_{*}$
is a projective $E_{*}$-module.  We note that the universal
coefficient spectral sequence
of~\cite[Theorem~IV.4.1]{elmendorf-kriz-mandell-may} implies that if $P$
is projective then the natural map 
\[
\cat{D} (E) (P,X) \xrightarrow{} \Hom_{E_{*}} (P_{*},X_{*})
\]
is an isomorphism for all $X\in \cat{D} (E)$.  The converse is also
true, for if this natural map is an isomorphism, then there are no
nonzero ghosts with domain $P$.  

The following lemma gives the most obvious relationship between ghost
dimension and Rouquier dimension. 

\begin{lemma}\label{lem-Rouquier}
Suppose $E$ is an $S$-algebra or an ordinary ring.  Then 
\[
\ghdim E\leq \Rdim E,
\]
with equality holding when $\ghdim E=0$.  
\end{lemma}

\begin{proof}
The inequality is clear.  If $\ghdim E=0$, then every compact object
is a retract of a coproduct of suspensions of $E$, so is also a
retract of a finite coproduct of suspensions of $E$.
\end{proof}

Note that $S$-algebras with $\ghdim E=0$ are called
\textbf{von Neumann regular}, because if $R$ is an ordinary ring,
$\ghdim R=0$ if and only if $R$ is von Neumann regular
(see~\cite{hovey-lockridge-semisimple}).  

The ghost dimension and the Rouquier dimension agree when $E_{*}$
is coherent, as we see in the following proposition.

\begin{proposition}\label{prop-coherent}
Suppose $E$ is an $S$-algebra or an ordinary ring for which $E_{*}$ is
coherent.  Then $X\in \class{E}^{n}_{f}$ if and only if $X\in
\class{E}^{n}$ and $X$ is compact. Thus $\ghdim E=\Rdim E$.  
\end{proposition}

There is no reason to think that $\ghdim E=\Rdim E$ if $E_{*}$ is not
coherent, even if $E$ is an ordinary ring $R$, but we do not know a
counterexample.

\begin{proof}
We first prove the well-known fact that, since $E_{*}$ is coherent,
for every compact object $X$ of $\cat{D} (E)$, $X_{*}$ is a finitely
presented $E_{*}$-module.  Consider the class $\cat{C}$ of $X$ for
which $X_{*}$ is a finitely presented $E_{*}$-module.  We claim that
$\cat{C}$ is a thick subcategory.  Given this, since $\cat{C}$
contains $E$, it contains all the compact objects as required (this is
well-known; see~\cite[Theorem~2.1.3]{hovey-axiomatic} for a general
approach to this fact).  To prove that $\cat{C}$ is thick, we must
show that it is closed under suspensions and retracts, which is
obvious in this case, and also if we have an exact triangle
\[
X \xrightarrow{f} Y \xrightarrow{g} Z \xrightarrow{h} \Sigma X
\]
in which $X,Y$ are in $\cat{C}$, then $Z\in \cat{C}$.  Given such an
exact triangle, $Z_{*}$ is an extension of $\coker f_{*}$ by $\ker
(\Sigma f_{*})$.  Finitely presented modules are always closed under
cokernels of morphisms and extensions~\cite[Lemma~4.54]{lam}.  If
$E_{*}$ is coherent, finitely presented modules are closed under
kernels of morphisms as well, and so $Z_{*}$ is finitely presented.
Indeed, if $g\mathcolon M\xrightarrow{}N$ is a morphism of finitely
presented modules over a coherent ring, then $M/\ker f\cong \im f$ is
a finitely generated submodule of the finitely presented module $N$,
so is finitely presented.  Hence $\ker f$ must be a finitely generated
submodule of the finitely presented module $M$, so is finitely
presented.

Now suppose $X\in \class{E}^{n}$ and $X$ is compact.  By induction,
because $X_{*}$ is finitely presented, we can choose finite coproducts
$P_{i}$ of suspensions of $E$ and exact triangles
\[
X_{i+1} \xrightarrow{} P_{i} \xrightarrow{f_{i}} X_{i}
\xrightarrow{\delta_{i}}\Sigma X_{i+1}
\]
where $X_{0}=X$ and $f_{i}$ is onto on homotopy, so $\delta_{i}$ is a
ghost.  Of course each $X_{i}$ is compact.  Then consider the exact
triangle
\[
X_{i+1}\xrightarrow{}Y_{i}\xrightarrow{}X
\xrightarrow{}\Sigma^{i+1}X_{i+1}.
\]
By using the $3\times 3$ lemma (well-known, but stated in
~\cite[Lemma~A.1.2]{hovey-axiomatic}) on the square 
\[
\begin{CD}
X @>>> \Sigma^{i}X_{i} \\
@| @VVV \\
X @>>> \Sigma^{i+1}X_{i+1}
\end{CD}
\]
we see that there is an exact triangle
\[
\Sigma^{i-1}P_{i}\xrightarrow{} Y_{i-1} \xrightarrow{}Y_{i} \xrightarrow{}
\Sigma^{i} P_{i}.
\]
Hence $Y_{i}$ is in $\class{E}^{i}_{f}$.  On the other hand, the map
$X\xrightarrow{}\Sigma^{n+1}X_{n+1}$ is the composite of $n+1$ ghosts,
so it is null since $X\in \class{E}^{n}$, using
Proposition~\ref{prop-proj-dim}.  Hence $X$ is a retract of $Y_{n}$,
so $X\in \class{E}^{n}_{f}$.
\end{proof}

\section{Flat dimension}\label{sec-flat}

We now offer a different approach to the weak dimension of an
$S$-algebra or an ordinary ring using flat modules.  As discussed in
the introduction, we did not use these in~\cite{hovey-lockridge-ghost}
because of the fundamental issue that homology functors are not always
representable in $\cat{D} (E)$ for an $S$-algebra $E$, or even a ring
$R$.

However, we can still define $\cat{F}$ to be the class of objects $F$
in $\cat{D} (E)$ such that $F_{*}$ is flat over $E_{*}$.  In this
case, we say that $F$ is \textbf{flat} (as an object of $\cat{D}
(E)$).  We can then define an $E$-module $X\in \cat{D} (E)$ to have
\textbf{constructible flat dimension $n$}, written $\cfdim X=n$, if
$X\in \class{\cat{F}}^{n}$.  Note that $\cfdim X\leq \pdim X$, since
every projective is flat.  We can then consider the maximal
constructible flat dimension of any object in $\cat{D} (R)$, or of
just a compact object in $\cat{D} (R)$.  Both of these are possible
candidates for something like weak dimension.  In principle, we could
also consider a definition similar to the Rouquier dimension, using
compact flat objects to resolve arbitrary compact objects, but we will
see that a compact flat object is projective, so this would just
recover Rouquier dimension.

\begin{proposition}\label{prop-flat-resolution}
We have $\cfdim X\leq \fdim X_{*}$ for all $X\in \cat{D} (E)$.  In
particular, the maximal constructible flat dimension of an object in
$\cat{D} (E)$ is bounded above by $\wdim E_{*}$.
\end{proposition}

\begin{proof}
There is nothing to prove if $\fdim X_{*}$ is infinite, so suppose
$\fdim X_{*}=n$.  Then by beginning a projective resolution of
$X_{*}$, we get an exact sequence of $E_{*}$-modules
\[
0 \xrightarrow{} F \xrightarrow{} P_{n-1} \xrightarrow{d_{n-1}} \dotsb
\xrightarrow{d_{1}} P_{0} \xrightarrow{d_{0}} X_{*} \xrightarrow{}0
\]
where each $P_{i}$ is projective over $E_{*}$ and $F$ is flat over $E_{*}$.
This gives us short exact sequences
\[
0 \xrightarrow{} K_{i+1} \xrightarrow{} P_{i} \xrightarrow{}
K_{i}\xrightarrow{}0
\]
for $i\leq n-1$, where $K_{i}=\ker d_{i-1}$, $K_{0}=X_{*}$, and
$K_{n}=F$.  Because the $P_{i}$ are projective, these short exact
sequences are uniquely realizable by exact triangles in $\cat{D} (E)$
\[
X_{i+1} \xrightarrow{} Q_{i}\xrightarrow{}X_{i} \xrightarrow{}\Sigma
X_{i+1}
\]
where $X_{0}=X$, $(X_{i})_{*}=K_{i}$ and $(Q_{i})_{*}=P_{i}$.  In more
detail, $P_{i}$ is a retract of a direct sum of copies of $E_{*}$.
Thus we can let $Q_{i}$ be the corresponding retract of a coproduct of
copies of $E$.  Then one checks that a map out of $Q_{i}$ to any
object $Y$ is equivalent to a map $P_{i}\xrightarrow{}Y_{*}$.  This
gives us exact triangles of the form 
\[
\Sigma^{i-1}X_{i} \xrightarrow{}Y_{i}\xrightarrow{}X \xrightarrow{}
\Sigma^{i}X_{i}
\]
for all $i$, where the last map is the composite of the maps
$\Sigma^{j}X_{j}\xrightarrow{}\Sigma^{j+1}X_{j+1}$.  Using the
$3\times 3$ lemma on the commutative square 
\[
\begin{CD}
X @>>> \Sigma^{i}X_{i} \\
@| @VVV \\
X @>>> \Sigma^{i+1}X_{i+1}
\end{CD}
\]
gives us exact triangles 
\[
\Sigma^{i-1}Q_{i} \xrightarrow{} Y_{i} \xrightarrow{} Y_{i+1}
\xrightarrow{}\Sigma^{i} Q_{i}
\]
for all $i$.  In particular, $\cfdim Y_{i}\leq i-1$.  Now the exact triangle
\[
\Sigma^{n-1}X_{n} \xrightarrow{} Y_{n} \xrightarrow{}X \xrightarrow{}
\Sigma^{n}X_{n}
\]  
shows that $\cfdim X\leq n$, since $(X_{n})_{*}$ is flat.
\end{proof}

We now give an alternative characterization of the flat objects in
$\cat{D} (E)$.  Recall that a \textbf{phantom} map in $\cat{D} (E)$ is
a map $f\mathcolon X\xrightarrow{}Y$ such that $fg=0$ for all
$g\mathcolon A\xrightarrow{}X$ where $A$ is compact.  We need the
following lemma. 

\begin{lemma}\label{lem-phantom}
Suppose $E$ is an $S$-algebra.  A map $f\mathcolon X\xrightarrow{}Y$
is phantom in $\cat{D} (E)$ if and only if $\pi_{*} (f\wedge_{E}Z)=0$
for all left $E$-modules $Z$ if and only if $\pi_{*} (f\wedge_{E}Z)=0$
for all compact left $E$-modules $Z$.  
\end{lemma}

\begin{proof}
Spanier-Whitehead duality implies that $\pi_{*} (f\wedge_{E}Z)=0$ for
all compact left $E$-modules $Z$ if and only if $\cat{D} (E) (W,f)=0$
for all compact right $E$-modules $W$, which of course is the
definition of $f$ being phantom.  It remains to show that, under this
condition, $\pi_{*} (f\wedge_{E}Z)=0$ for all left $E$-modules $Z$.
But $\pi_{*} (-\wedge_{E}Z)$ is a homology theory on $\cat{D} (E)$,
and phantom maps vanish on all homology
theories~\cite[Proposition~1.1]{christensen-strickland}.  The reader
should note that Christensen and Strickland are working in the
ordinary category of spectra, where Spanier-Whitehead duality is
internal, but the same proof will work for a general $S$-algebra $E$
as long as we remember that Spanier-Whitehead duality shifts from left
to right $E$-modules and vice versa.
\end{proof}

\begin{proposition}\label{prop-Brown}
Suppose $E$ is an $S$-algebra or a ring, and $X\in \cat{D} (E)$.  Then
the following are equivalent\uc
\begin{enumerate}
\item $X_{*}$ is a flat $E_{*}$-module. 
\item Every ghost with domain $X$ is phantom.  
\item There is an exact triangle 
\[
P \xrightarrow{} X \xrightarrow{g} Y \xrightarrow{} \Sigma P
\]
where $P$ is projective and $g$ is phantom.  
\item Every map $A\xrightarrow{}X$, where $A$ is compact, factors
through a compact projective object.  
\item The natural map 
\[
X_{*} \otimes_{E_{*}} Z_{*} \xrightarrow{} \pi_{*} (X\wedge_{E}Z)
\]
is an isomorphism for all left $E$-modules $Z$.  
\end{enumerate}
\end{proposition}

If $E$ is an ordinary ring $R$, then $X\wedge_{E}Z$ would be the total
left derived tensor product of the chain complex $X$ of right
$R$-modules and the chain complex $Z$ of left $R$-modules.  

\begin{proof}
For any $X$, there is an exact triangle 
\[
P \xrightarrow{} X \xrightarrow{g} Y \xrightarrow{} \Sigma P
\]
in which $P$ is projective and $g$ is a ghost.  Indeed, we simply take
an epimorphism from a free $E_{*}$-module $P_{*}$ to $X_{*}$.  We
then let $P$ be the corresponding coproduct of suspensions of $E$,
which is projective, and realize the map $P_{*}\xrightarrow{}X_{*}$ as
a map $P\xrightarrow{}X$.  The cofiber $g$ is then automatically a
ghost, and every other ghost with domain $X$ factors through $g$.  

It follows from this that part~(2) and~(3) are equivalent.  This also
means that part~(3) implies part~(4), since part~(3) means that a map
$A\xrightarrow{}X$, where $A$ is compact, factors through a
projective, and therefore a free $E$-module.  Since $A$ is compact, it
must factor through a finite coproduct of suspensions of $E$.

We now show that part~(4) implies part~(1).  Recall the filtered
category $\Lambda (X)$ from~\cite[Section~2.3]{hovey-axiomatic} of
maps from a compact object into $X$, and consider the full subcategory
$\Lambda ' (X)$ of maps from a compact projective into $X$.  Given
part~(4), $\Lambda ' (X)$ is cofinal in $\Lambda (X)$ and itself
filtered.  Thus, for any homology theory $H$, $H (X)=\colim_{\Lambda '
(X)}H (P_{\alpha})$ by~\cite[Corollary~2.3.11]{hovey-axiomatic}.  In
particular, $X_{*}$ is a colimit of finitely generated projective
modules, so is flat.

To see that part~(1) implies part~(5), use the universal coefficient
spectral sequence 
\[
\Tor^{E_{*}}_{s,t} (X_{*},Z_{*}) \Rightarrow \pi_{t-s} (X\wedge_{E}Z)
\]
of~\cite[Theorem~IV.4.1]{elmendorf-kriz-mandell-may}.  

To see that part~(5) implies part~(2), suppose that $g$ is a ghost
with domain $X$.  Part~(5) then implies that $\pi_{*} (g\wedge_{E}Z)=0$
for all left $E$-modules $Z$, which implies that $g$ is phantom by
Lemma~\ref{lem-phantom}.
\end{proof}

Recall that, for a general ring $R$, there are finitely generated flat
modules which are not projective, though of course every finitely
presented flat module is projective.  The following corollary is an
analog of this fact for $S$-algebras.  

\begin{corollary}\label{cor-Brown}
Suppose $E$ is an $S$-algebra or an ordinary ring.  If $X$ is a
compact flat object of $\cat{D} (E)$, then $X$ is projective.  
\end{corollary}

\begin{proof}
The universal ghost out of $X$ is phantom, and hence null.  Thus $X$
is projective, necessarily finitely generated since $X$ is compact.  
\end{proof}

We have been unable to fully generalize Proposition~\ref{prop-Brown} to
objects $X$ with $\cfdim X=n$.  We therefore make the following
definition.  

\begin{definition}\label{defn-flat-dim-again}
Suppose $E$ is an $S$-algebra or an ordinary ring, and $X\in \cat{D}
(E)$.  We say that $X$ has \textbf{flat dimension} at most $n$,
written $\fdim X\leq n$, if every composite of $n+1$ ghosts with
domain $X$ is phantom.
\end{definition}

We then have the following theorem.  

\begin{theorem}\label{thm-flat}
Suppose $E$ is an $S$-algebra or an ordinary ring, and $X\in \cat{D}
(E)$.  Then $\fdim X\leq \cfdim X$, and the following are equivalent.  
\begin{enumerate}
\item $\fdim X\leq n$.  
\item There is an exact triangle 
\[
B \xrightarrow{} X \xrightarrow{g} Y \xrightarrow{} \Sigma B
\]
where $\pdim B\leq n$ and $g$ is phantom.  
\item Every map $A\xrightarrow{}X$, where $A$ is compact, factors
through a compact object $B$ with $\pdim B\leq n$.  
\item For any left $E$-module $Z$, in the universal coefficient
spectral sequence
\[
E^{2}_{s,t} = \Tor^{E_{*}}_{s,t} (X_{*},Z_{*}) \Rightarrow \pi_{t-s}
(X\wedge_{E}Z),
\]
we have $E^{\infty}_{s,*}=0$ for all $s>n$.  
\end{enumerate}
\end{theorem}

It would be good to find an example where $\fdim X<\cfdim X$, or to
prove they are always equal.  

\begin{proof}
We show that $\cfdim X\leq n$ implies every composition of $n+1$ ghosts
out of $X$ is phantom by induction on $n$.  The base case of $n=0$ is
Proposition~\ref{prop-Brown}.  For the induction step, suppose $\cfdim
X\leq n$, 
\[
X \xrightarrow{g_{1}} Z_{1} \xrightarrow{g_{2}} \dotsb
\xrightarrow{g_{n+1}} Z_{n+1}
\]
is the composition $g$ of $n+1$ ghosts, and $f\mathcolon
A\xrightarrow{}X$ is a map from a compact object.  We must show
$gf=0$.  Since $\cfdim X\leq n$, there is a cofiber sequence 
\[
F \xrightarrow{}Y \xrightarrow{s}\widetilde{X} \xrightarrow{h} \Sigma F
\]
where $F$ is flat, $\cfdim Y\leq n-1$, and there are maps 
\[
X\xrightarrow{i} \widetilde{X}\xrightarrow{r} X
\]
with $ri=1_{X}$.  Since $gf= (gr) (if)$, and $gr$ is again a
composition of $n+1$ ghosts, we can assume $X=\widetilde{X}$.  

The composition $hf$ factors through a finitely generated projective
$P$, by Proposition~\ref{prop-Brown}.  This gives us a commutative
diagram 
\[
\begin{CD}
\Sigma^{-1}P @>>> B @>t>> A @>u>> P \\
@VVV @Vf'VV @VfVV @VVV \\
F @>>> Y @>>s> X @>>h> \Sigma F
\end{CD}
\]
whose rows are exact triangles.  Note that $B$ is necessarily a
compact object, and so the composition $g_{n}\circ \dotsb \circ
g_{1}sf'$ is null, since $\fdim Y\leq n-1$.  Hence we have 
\[
g_{n}\circ \dotsb \circ g_{1}ft=0 \text{ so } g_{n}\circ \dotsb \circ
g_{1}f=vu
\]
for some map $v\mathcolon P\xrightarrow{}Z_{n}$.  But then
$g_{n+1}v=0$, since $P$ is projective, and so 
\[
g_{n+1}\circ \dotsb \circ g_{1}f=0
\]
as required.  This completes the proof that $\fdim X\leq \cfdim X$.  

The work of Christensen~\cite[Theorem~3.5]{christensen} implies that,
for any $X$, there is an exact triangle
\[
B \xrightarrow{} X \xrightarrow{g} Y \xrightarrow{} \Sigma B
\]
with $\pdim B\leq n$ and $g$ is a composite of $n+1$ ghosts.  But then
every composite of $n+1$ ghosts with domain $X$ factors through $g$.
Thus every composite of $n+1$ ghosts is phantom if and only if $g$ is
phantom, so part~(1) and part~(2) are equivalent.  

Now, the universal coefficient spectral sequence for $\pi_{*}
(X\wedge_{E}Z)$ is constructed as follows.  Beginning with $X_{0}=X$,
we construct exact triangles
\[
X_{i+1} \xrightarrow{} Q_{i}\xrightarrow{h_{i}}X_{i} \xrightarrow{k_{i}}\Sigma
X_{i+1}
\]
as in the proof of Proposition~\ref{prop-flat-resolution}, in which
$Q_{i}$ is projective, $h_{i}$ is onto on homotopy, and $k_{i}$ is a
ghost.  We then smash them with $Z$ and take homotopy to get our
spectral sequence of homological type.  An element in $\pi_{*}
(X\wedge_{E}Z)$ is detected in $E^{\infty}_{s,*}$ if and only if it is
in the kernel of
\[
\pi_{*} (X\wedge_{E}Z)\xrightarrow{}\pi_{*} (\Sigma^{j}X_{j}\wedge_{E}Z)
\]
for $j=s$ but not for $j=s-1$.  So we must determine when the map 
\[
\pi_{*} (X\wedge_{E}Z)\xrightarrow{}\pi_{*} (\Sigma^{n}X_{n}\wedge_{E}Z)
\]
is zero for all $Z$.  However, comparison of this construction
of~\cite[Section~IV.5]{elmendorf-kriz-mandell-may}
with~\cite[Theorem~3.5]{christensen} shows that in fact the map
$X\xrightarrow{}\Sigma^{n}X_{n}$ is the same as the universal
composite of $n+1$ ghosts out of $X$, $g\mathcolon
X\xrightarrow{}Y$ of the previous paragraph.  Lemma~\ref{lem-phantom}
then shows part~(2) and part~(4) are equivalent.

It is clear that if every map from a compact to $X$ factors through a
$Y$, compact or not, with $\pdim Y\leq n$, then every composite of
$n+1$ ghosts out of $X$ is phantom, so part~(3) implies part~(1).  For
the converse, in view of part~(2), it suffices to prove that every map
from a compact object to an $A$ with $\pdim A\leq n$ factors through a
compact $B$ with $\pdim B\leq n$.  we proceed by induction on $n$.
The base case $n=0$ is implied by Proposition~\ref{prop-Brown}.  For
the induction step, suppose we have a map $f\mathcolon
F\xrightarrow{}A$, where $F$ is compact and $\pdim A\leq n+1$.  Then
we have an exact triangle
\[
C\xrightarrow{r} P \xrightarrow{s} \widetilde{A} \xrightarrow{t} \Sigma C
\]
with $P$ projective, $\pdim C\leq n$, and $A$ is a retract of
$\widetilde{A}$.  It suffices to show that the composite 
\[
F\xrightarrow{}A\xrightarrow{}\widetilde{A}
\]
factors through a compact $B$ with $\pdim B\leq n+1$.  We can
therefore assume $A=\widetilde{A}$.  By the induction hypothesis, we
can write $tf=\phi h$ for some $h\mathcolon F\xrightarrow{}D$, where
$D$ is compact with $\pdim D\leq n$.  This gives us a map of exact
triangles 
\[
\begin{CD}
\Sigma^{-1}D @>>> Z @>>> F @>h>> D   \\
@V\Sigma^{-1}\phi VV @V\psi VV @VfVV @VV\phi V \\
C @>>r> P @>>s> A @>>t> \Sigma C
\end{CD}
\]
But then $Z$ is necessarily compact, so we can write $\psi=j\tau$,
where the codomain $Q$ of $\tau$ is compact projective.  By taking the
weak pushout (which amounts to applying the $3\times 3$ lemma), we get
the following diagram, whose rows are exact triangles. 
\[
\begin{CD}
\Sigma^{-1}D @>>> Z @>>> F @>h>> D   \\
@| @V\tau VV @V\sigma VV @| \\
\Sigma^{-1}D @>>> Q @>>> B' @>>> D \\
@V\Sigma^{-1}\phi VV @VjVV @V\rho VV @VV\phi V \\
C @>>r> P @>>s> A @>>t> \Sigma C
\end{CD}
\]
Now $B'$ is a compact object with $\pdim B'\leq n+1$, but
unfortunately $\rho \sigma$ need not be equal to $f$.  Nevertheless,
we do have
\[
t\rho \sigma =\phi h = tf,
\]
so $f-\rho \sigma =sq$ for some map $q\mathcolon F\xrightarrow{}P$.
But then $q=iq'$ for some map $q'\mathcolon F\xrightarrow{}Q'$, where
$Q'$ is compact projective.  Altogether then, we have 
\[
f=\rho \sigma +siq'.
\] 
This means that $f$ factors through the compact object $B'\amalg Q'$.
Indeed, $f$ is the composite 
\[
F \xrightarrow{(\sigma ,q')} B'\amalg Q' \xrightarrow{\rho + si} A.
\]
Since $\pdim (B'\amalg Q')\leq n+1$, the proof is complete. 
\end{proof}

\begin{corollary}\label{cor-compact}
Suppose $E$ is an $S$-algebra or an ordinary ring.  If $X$ is a
compact object of $\cat{D} (E)$, then 
\[
\fdim X= \cfdim X =\pdim X.
\]
In particular, $\ghdim E$ is the maximal \ulp constructible or not\urp
\ flat dimension of a compact object of $\cat{D} (E)$, or $\infty$ if
there is no such maximal dimension.
\end{corollary}

\begin{proof}
We always have $\fdim X\leq \cfdim X\leq \pdim X$.  Suppose $\fdim
X=n$ and $X$ is compact.  Then every composition of $n+1$ ghosts out
of $X$ is phantom, hence null.  Thus $\pdim X\leq n$, so $\pdim X=\fdim X$.
\end{proof}

\begin{corollary}\label{cor-maximal}
Suppose $E$ is an $S$-algebra or an ordinary ring.  Then $\ghdim
E=\sup \fdim X$ as $X$ runs through arbitrary objects of $\cat{D}
(E)$.   
\end{corollary}

We do not know whether this corollary remains true for the
constructible flat dimension.  

\begin{proof}
Corollary~\ref{cor-compact} implies that 
\[
\ghdim E\leq \sup_{X} \fdim X.
\]
Now suppose $\ghdim E=n$, so that $\fdim F\leq n$ for all compact
objects $F$ of $\cat{D} (E)$.  Choose an arbitrary $X\in \cat{D} (E)$,
and consider the universal coefficient spectral sequence
\[
E^{2}_{s,t}=\Tor^{E_{*}}_{s,t} (X_{*},Z_{*})\Rightarrow \pi_{t-s}
(X\wedge_{E}Z)
\]
for an arbitrary left $E$-module $Z$.  We must show that
$E^{\infty}_{s,*}=0$ for $s>n$.  Since this spectral sequence is of
homological type, this means we must show that every element in
$\pi_{*} (X\wedge_{E}Z)$ is in filtration $n$ (and possibly lower
filtration as well).  But the functor $\pi_{*}
(-\wedge_{E}Z)$ is a homology theory on right $E$-modules, so 
\[
\pi_{*} (X\wedge_{E}Z) =\colim \pi_{*} (F\wedge_{E}Z),
\]
where the colimit is taken over all maps $F\xrightarrow{}X$ from a
compact object to $X$.  Then any element of $\pi_{*} (X\wedge_{E}Z)$
comes from some $\pi_{*} (F\wedge_{E}Z)$, where it lies in filtration
$n$.  Naturality of the spectral sequence implies that it also lies in
filtration $n$ in $\pi_{*} (X\wedge_{E}Z)$.
\end{proof}

We now point out another advantage of the ghost dimension; it is
left-right symmetric, like the usual weak dimension of rings.

\begin{theorem}\label{thm-weak-symmetric}
Suppose $E$ is an $S$-algebra or an ordinary ring.  Then 
\[
\ghdim E = \ghdim E^{\textup{op}}.
\]
\end{theorem}

\begin{proof}
The ghost dimension of $E$ is the largest $n$ such that there exists a
right $E$-module $X$ and a left $E$-module $Y$ for which
$E^{\infty}_{n,*}$ is nonzero in the universal coefficient spectral
sequence for $\pi_{*} (X\wedge_{E}Y)$.  This is obviously left-right
symmetric.  
\end{proof}

Summing up, then, we are left with three possible definitions for the
weak dimension of an $S$-algebra $E$.  We list the basic inequalities
between these definitions in the following theorem, which also proves
the main conjecture of~\cite{hovey-lockridge-ghost} that $\ghdim
R=\wdim R$ for ordinary rings $R$.  

\begin{theorem}\label{thm-summary}
Suppose $E$ is an $S$-algebra or an ordinary ring.  Then we have 
\begin{gather*}
\ghdim E = \sup_{X \textup{compact}} \cfdim X = \sup_{X \textup{compact}}
\pdim X \\
= \sup_{X \textup{compact}} \fdim X = \sup_{X \textup{arbitrary}}
\fdim X.  
\end{gather*}
Furthermore, 
\[
\ghdim E \leq \sup_{X \textup{arbitrary}} \cfdim X \leq \wdim E_{*}
\]
with equality if $E$ is an ordinary ring.  Finally, 
\[
\ghdim E \leq \Rdim E
\]
with equality if $E_{*}$ is coherent.  
\end{theorem}

\begin{proof}
The only thing we have not already proved is that equality holds in
the first chain of inequalities when $R$ is an ordinary ring.  But the
main result of~\cite{hovey-lockridge-ghost} is that $\wdim R\leq
\ghdim R$, giving us the desired equalities.  
\end{proof}

% \bibliography{hovey}
% \bibliographystyle{amsalpha}
\providecommand{\bysame}{\leavevmode\hbox to3em{\hrulefill}\thinspace}
\providecommand{\MR}{\relax\ifhmode\unskip\space\fi MR }
\providecommand{\MRhref}[2]{%
  \href{http://www.ams.org/mathscinet-getitem?mr=#1}{#2}
}
\providecommand{\href}[2]{#2}

\end{document}